\documentclass[12pt,oneside]{amsart}
\usepackage[T2A]{fontenc}
\usepackage[cp1251]{inputenc}

\usepackage{amsmath,amssymb,amsthm}
\usepackage[mathscr]{eucal}

\usepackage{graphics}
\usepackage{tikz}

\theoremstyle{plain}
\newtheorem{theorem}{Theorem}
\newtheorem{lemma}{Lemma}

\theoremstyle{remark}

\begin{document}
	
	\begin{abstract}
		It is known that an ideal triangulation of a compact $3$-manifold with nonempty boundary is minimal if and only if it contains the minimum number of edges among all ideal triangulations of the manifold. Therefore, any ideal one-edge triangulation (i.e., an ideal singular triangulation with exactly one edge) is minimal. Vesnin, Turaev, and the first author showed that an ideal two-edge triangulation is minimal if no $3$-$2$ Pachner move can be applied. In this paper we show that any of the so-called poor ideal three-edge triangulations is minimal. We exploit this property to construct minimal ideal triangulations for an infinite family of hyperbolic $3$-manifolds with totally geodesic boundary.
	\end{abstract}
	
	\title[Poor ideal three-edge triangulations]{Poor ideal three-edge triangulations are minimal}
	
	\author{Evgeny Fominykh, Ekaterina Shumakova}
	
	\address{Saint Petersburg University; St. Petersburg Department of Steklov Mathematical Institute of Russian Academy of Sciences, Saint Petersburg, Russia \\
	efominykh@gmail.com}
	
	\address{Saint Petersburg University, Saint Petersburg;  Chelyabinsk State University, Chelyabinsk, Russia\\
	shumakova\_kate@mail.ru}
	
	\thanks{This work is supported by the Russian Science Foundation under grant 19-11-00151.} 
	

\maketitle

\section{Introduction}

	In this paper we restrict ourselves to considering connected compact $3$-manifolds $M$ with nonempty boundary and their ideal triangulations. An ideal triangulation of $M$ is called \emph{minimal} if it contains the minimum number of tetrahedra among all ideal triangulations of $M$. Only a few infinite series of minimal ideal triangulations are known. A geometric approach for proving the minimality of ideal triangulations of hyperbolic $3$-manifolds with cusps is build on a volume-based lower bound for the number of tetrahedra \cite{Anis05, FomGarGoTarVes16, IshNem16}. A topological approach to this problem is based on the following simple idea. An ideal triangulation of a compact $3$-manifold with nonempty boundary is minimal if and only if it contains the minimum number of edges among all ideal triangulations of the manifold. Therefore, ideal triangulations with a single edge are automatically minimal. It is shown in \cite{FrigMarPet03-1} that any compact $3$-manifold having an ideal triangulation with a single edge is hyperbolic with totally geodesic boundary, and infinite series of such triangulations are constructed. Vesnin, Turaev, and the first author showed \cite{VesTurFom16} that an ideal two-edge triangulation is minimal if no $3$-$2$ Pachner move can be applied. Several infinite series of such triangulations of hyperbolic $3$-manifolds with totally geodesic boundary were constructed in \cite{PaolZim96, VesFom11, VesTurFom15}. The minimality of ideal three-edge triangulations that are analogous to triangulations from \cite{PaolZim96} and determine hyperbolic $3$-manifolds with totally geodesic boundary is proved in \cite{VesFom12}.
	
	In this paper we prove (Theorem \ref{thm:main}) that any so-called poor ideal three-edge triangulation is minimal. We construct a three-parameter family of such triangulations (Theorem \ref{thm:example}) and prove that if all parameters have the same values, the corresponding $3$-manifolds are hyperbolic with totally geodesic boundary (Theorem \ref{thm:hyperbolic}). The so-called $\varepsilon$-invariant of $3$-manifolds \cite{MatOvchSok}, which is a homologically trivial part of the Turaev -- Viro invariant \cite{Turaev-Viro} of order $5$ play a key role in our proof of Theorem \ref{thm:main}.
	
\section{Poor triangulations}

	In this section we describe the duality between ideal triangulations and special spines of $3$-manifolds with nonempty boundary and also introduce a concept of a poor triangulation.
	
\subsection{Ideal triangulations}
	
	Let $\mathcal{T}$ be a cell complex made from finitely many $3$-simplices by gluing all of their $2$-dimensional faces in pairs via simplicial maps, and $|\mathcal{T}|$ its underlying space. The simplices prior to identification, and their vertices, edges, and faces, are called \emph{model cells}. Let $\mathcal{T}^{(0)}$ denote the $0$-skeleton of $\mathcal{T}$.
   
	Note that $\mathcal{T}$ may actually not be a $3$-manifold, because the link of a vertex could be any closed surface, and the link of the midpoint of an edge could be a projective plane. Let $M$ be a compact $3$-manifold with nonempty boundary. If $|\mathcal{T}|\setminus |\mathcal{T}^{(0)}|$ is homeomorphic to the interior of $M$, then we say that $T$ is an \emph{ideal (singular) triangulation} of $M$.

\subsection{Special spines}
	
	Recall the definitions of simple and special polyhedra (we follow the book \cite{MatBook}). A compact polyhedron $P$ is said to be \emph{simple} if the link of each of its points $x \in P$ is homeomorphic to one of the following one-dimensional polyhedra: (a) a circle (in this case the point $x$ is said to be \emph{nonsingular}); (b) a union of a circle and a diameter (in this case the point $x$ is said to be \emph{triple}); (c) a union of a circle and three radii (in this case the point $x$ is called a \emph{true vertex}). Typical neighborhoods of points of a simple polyhedron are shown in Fig. \ref{fig:neighborhood}. We will call the points of types (b) and (c) \emph{singular} points. The set of singular points of a simple polyhedron $P$ is called its \emph{singular graph} and is denoted by $SP$. In the general case, the set $SP$ is not a graph on the set of true vertices of $P$, because $SP$ may contain closed triple lines without true vertices. If there are no closed triple lines, then $SP$ is a 4-regular graph (each vertex is incident to exactly four edges) that may have loops and multiple edges.
	
	\begin{figure}[ht]
		\begin{center}
			\includegraphics[scale=0.8]{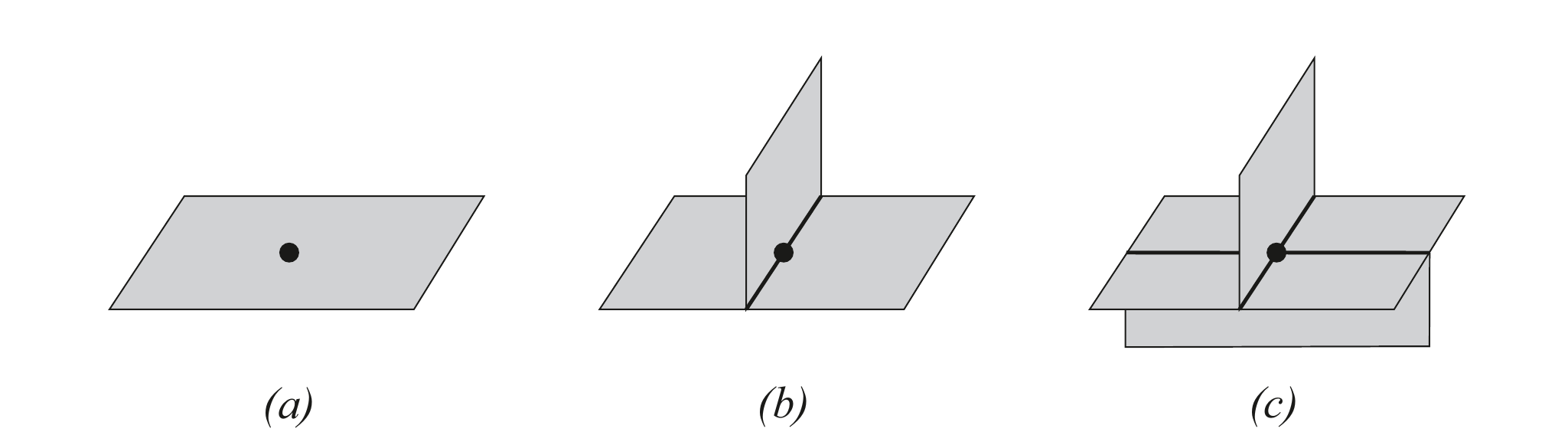}
		\end{center}
		\caption{Allowable neighborhoods in a simple polyhedron.} 
		\label{fig:neighborhood}
	\end{figure}
	
	Every simple polyhedron has a natural \emph{stratification}. The strata of dimension $2$ (\emph{$2$-components}) are the connected components of the set of nonsingular points. The strata of dimension $1$ are the open or closed triple lines. The strata of dimension $0$ are the true vertices. It is natural to require that each stratum be a cell. A simple polyhedron $P$ is said to be \emph{special} if the following conditions are satisfied:\\
   (1) each one-dimensional stratum of $P$ is an open $1$-cell. \\
   (2) each $2$-component of the polyhedron $P$ is an open $2$-cell.\\ It is obvious that if $P$ is connected and contains at least one true vertex, then condition (2) implies condition (1).
	
	Let $M$ be a connected compact $3$-manifold with nonempty boundary. A compact polyhedron $P\subset M$ is called a \emph{spine} of $M$ if $M\setminus P$ is homeomorphic to $\partial M \times (0,1]$. A spine of $M$ is said to be \emph{special} if it is a special polyhedron. It is easy to see that a special spine of a connected $3$-manifold is connected.

\subsection{Duality between ideal triangulations of $3$-manifolds with nonempty boundary and their special spines}

	Now we describe a relation between ideal triangulations and special spines of compact $3$-manifolds with nonempty boundary.
	
	Each ideal triangulation $\mathcal{T}$ of a $3$-manifold $M$ naturally determines a dual special spine $P$. Namely, let $\mathcal{T}$ be obtained by gluing model tetrahedra $T_{1}, \ldots, T_{n}$. In each tetrahedron $T_{i}$ consider the union $R_{i}$ of the links of all four vertices of $T_{i}$ in its first barycentric subdivision (see Fig.~\ref{fig:dualspine}). Gluing the tetrahedra $T_{1}, \ldots, T_{n}$ induces gluing the corresponding polyhedra $R_{1}, \ldots, R_{n}$ together. We get a special polyhedron $P = \cup_{i=1}^{n} R_{i}$ which is a spine of $M$. It is well known \cite{MatBook} that the correspondence $\mathcal{T}\to P$ induces a bijection between (the equivalence classes of) ideal triangulations of $M$ and (the homeomorphism classes of) special spines of $M$. Moreover, the correspondence $\mathcal{T}\to P$ induces one-to-one correspondences between the edges, faces, and tetrahedra of $\mathcal{T}$ and the $2$-components, edges, and true vertices of $P$ which preserves the incidence relations.
	
	\begin{figure}[ht]
		\begin{center}
			\includegraphics[scale=0.5]{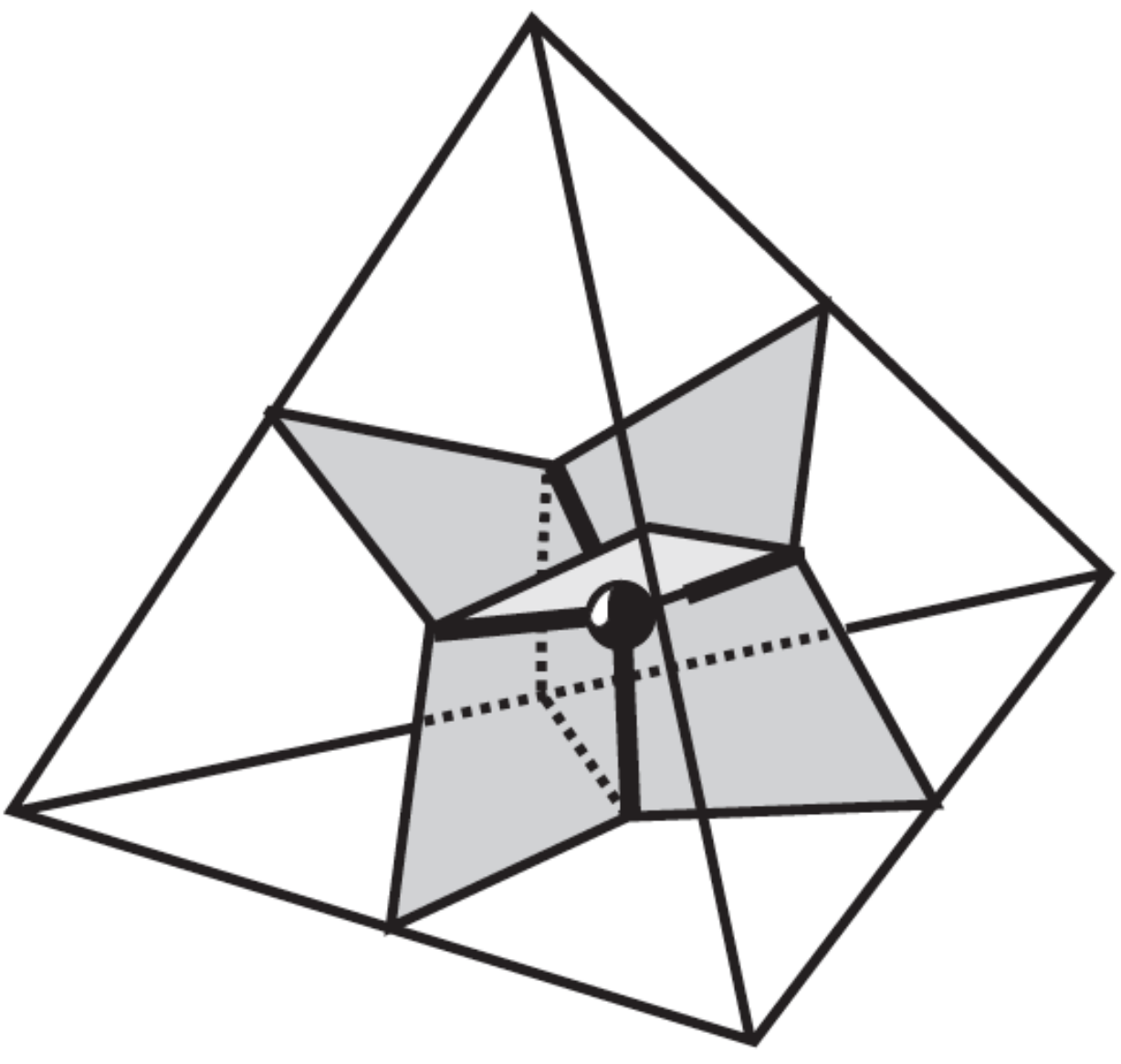}
		\end{center}
		\caption{Duality between ideal triangulations and special spines.}
		\label{fig:dualspine}
	\end{figure}

	Let us describe the inverse procedure that constructs an ideal triangulation $\mathcal{T}$ of $M$ from a special spine $P$ of $M$. Choose a point $a_\xi$ inside every $2$-component $\xi$ of $P$ and a point $b_\gamma$ inside every edge $\gamma$ of $P$. Recall that there is a characteristic map $f: D^{2} \to P$, which carries the interior of $D^{2}$ onto $\xi$ homeomorphically and whose restriction to $S^{1} = \partial D^{2}$ is a local embedding. We will call the curve $f|_{\partial D^{2}} : {\partial D^{2}} \to P$ (and its image $f|_{\partial D^{2}} (\partial D^{2})$) the \emph{boundary curve} of $\xi$ and denote it by $\partial \xi$. Connect each point $a_\xi$ by arcs in $\xi$ with all those points $b_\gamma$ that lie on the boundary curve of $\xi$. Note that each $b_\gamma$ is incident to exactly three arcs. We get a decomposition of $P$ into the union $\cup_j E_j $ of several copies of the polyhedron $E$ which is the regular neighborhood of a true vertex (see Fig. \ref{fig:neighborhood}(c)). In other words, $P$ is obtained by gluing together all the polyhedra $E_j$. Now each $E_j$ we embed in the model tetrahedron $T_{j}$ as shown in Fig \ref{fig:dualspine}.
	Gluing the $E_j$'s determines the pairwise identifications of the model faces of $T_{j}$'s.

\subsection{Poor spines and triangulations}

	Let $\mathcal{F}(P)$ be the set of all simple subpolyhedra of a special polyhedron $P$ including $P$ and the empty set. We call the polyhedron $P$ \emph{poor} if it does not contain proper simple subpolyhedra, i.e. $\mathcal{F}(P) = \{\emptyset, P\}$. Let $\mathcal{T}$ be an ideal triangulation of a connected compact $3$-manifold with nonempty boundary. The triangulation $\mathcal{T}$ will be called \emph{poor} if its dual special polyhedron is poor.
	
	For each edge $e$ of $\mathcal{T}$ denote by $\mathcal{E}(e)$ the set of all model edges that are identified to form $e$. Then the notion of poor triangulation can be expressed in terms of the triangulation itself as follows. 

\begin{lemma}
	\label{lm:_poor_spine}
	An ideal triangulation is poor if and only if for any proper subset $S$ of the set of all its edges there is a model face that is incident to exactly one model edge in $ \bigcup_{e\in S}\mathcal{E}(e)$.
\end{lemma}
\begin{proof}
	Let $\mathcal{T}$ be a connected ideal triangulation and $P$ the dual special polyhedron corresponding to $\mathcal{T}$. Since any simple subpolyhedron of $P$ is compact, if it contains a point of a $2$-component, then it contains the whole of it. Thus, to describe a simple subpolyhedron of $P$ it is enough to indicate which $2$-components of $P$ it includes (its triple points and true vertices will then be determined uniquely). It is easy to see that the given set of $2$-components does not determine a simple subpolyhedron of $P$ if and only if there is an edge $\ell$ in $P$ such that the boundary curves of the selected $2$-components pass through $\ell$ exactly once. This is equivalent to the fact that the model face dual to $\ell$ is incident to exactly one model edge in $\bigcup_{e\in S} \mathcal{E}(e)$, where $S$ is the set of edges of $\mathcal{T}$ that are dual to the selected $2$-components.
\end{proof}

\section{Main theorem}
		
	\begin{theorem}
		\label{thm:main}
		Any poor ideal three-edge triangulation of a connected compact $3$-manifold with nonempty boundary is minimal.
	\end{theorem}
	
   The proof of Theorem \ref{thm:main} requires a certain amount of preparatory work. For an arbitrary ideal triangulation $\mathcal{T}$ we denote by $\mathtt{e}(\mathcal{T})$ and $\mathtt{t}(\mathcal{T})$ the number of edges and the number of tetrahedra in $\mathcal{T}$. Similarly, for an arbitrary special polyhedron $P$ we denote by $\mathtt{d}(P)$ and $\mathtt{v}(P)$ the number of $2$-components and the number of true vertices in $P$.
	
	\begin{lemma}
		\label{lm:euler}
		Let $\mathcal{T}$ be an ideal triangulation of a connected compact $3$-manifold $M$ with nonempty boundary. Then
		$$\chi(M) = \mathtt{e}(\mathcal{T}) - \mathtt{t}(\mathcal{T}).$$
	\end{lemma}
	\begin{proof}
		Let $P$ be a special spine of $M$ dual to $\mathcal{T}$.  Homotopy equivalence between $M$ and its spine $P$ implies the equality $\chi(M) = \chi (P)$. By definition, the number of edges of $P$ is equal to twice the number of its true vertices. Hence, $\chi(P) = \mathtt{d}(P) - \mathtt{v}(P)$. Finally, $\mathtt{d}(P) = \mathtt{e}(\mathcal{T})$ and $\mathtt{v}(P) = \mathtt{t}(\mathcal{T})$ due to the duality between $P$ and $\mathcal{T}$.
	\end{proof}	
	
	\begin{proof}[Proof of Theorem \ref{thm:main}]
		Let $\mathcal{T}_3$ be a poor ideal three-edge triangulation of a connected compact $3$-manifold $M$ with nonempty boundary. Suppose that $\mathcal{T}_3$ consists of exactly $n$ tetrahedra. Denote by $P_3$ the special spine of $M$ dual to $\mathcal{T}_3$.		
		
		On the contrary, suppose that $\mathcal{T}_3$ is not minimal. Then by Lemma \ref{lm:euler} two cases are possible: a minimal ideal triangulation of $M$ either consists of exactly $n-2$ tetrahedra and has exactly one edge, or consists of exactly $n-1$ tetrahedra and has exactly two edges.
		
		Consider the first case and denote by $\mathcal{T}_1$ a minimal ideal triangulation of $M$. Let $P_1$ be a special spine of $M$ dual to $\mathcal{T}_1$. Then $P_1$ has exactly $n-2$ true vertices and one $2$-component.
		
		Now we use the $\varepsilon$-invariant of $3$-manifolds, which is the homologically trivial part of order $5$ Turaev -- Viro invariant \cite{Turaev-Viro}. We give the definition of the $\varepsilon$-invariant following~\cite{MatBook}. Let $R$ be a special spine of a compact $3$-manifold $N$ with nonempty boundary. Set $\varepsilon = (1+\sqrt{5})/2$, a solution to the equation $\varepsilon^2=\varepsilon+1$. With each $Q\in \mathcal{F}(R)$ we associate its $\varepsilon$-weight by the formula
		$$
		w_{\varepsilon}(Q) = (-1)^{\mathtt{v}(Q)}\varepsilon^{\chi(Q)-\mathtt{v}(Q)},
		$$
		where $\mathtt{v}(Q)$ is the number of true vertices of the simple polyhedron $Q$ and $\chi(Q)$ is its Euler characteristic.  Set 
		$$
		t(R) = \sum_{Q\in \mathcal{F}(R)} w_{\varepsilon}(Q).
		$$
		As was shown in~\cite{MatBook}, if $R_1$ and $R_2$ are two special spines of the same $3$-manifold $N$, then $t(R_1)=t(R_2)$. Therefore, the $\varepsilon$-invariant $t(N)$ of $N$ is given by the formula $t(N)=t(R)$.
				
		We return to the proof of the theorem and calculate the value of $t(M)$, using the polyhedra $P_3$ and $P_1$. By lemma \ref{lm:_poor_spine}, $P_3$ is poor, which implies $\mathcal{F}(P_3) = \{\emptyset, P_3\}$. Since $P_1$ has exactly one $2$-component, we see that $\mathcal{F}(P_1) = \{\emptyset, P_1\}$. Thus
		$$
		t(M) = (-1)^{n} \varepsilon^{3-2n} + 1 = (-1)^{n-2} \varepsilon^{5-2n}  + 1, 
		$$ 
		which implies $\varepsilon^2 = 1$ and gives a contradiction.
		
		Now consider the more complicated second case. Denote by $\mathcal{T}_2$ a minimal triangulation of $M$. Let $P_2$ be a special spine of $M$ dual to $\mathcal{T}_2$. Then $P_2$ has exactly $n-1$ true vertices and two $2$-components. Let us describe all proper simple subpolyhedra of $P_2$.
		
		\begin{lemma} 
			\label{lm:proper_subpolyh}
			The polyhedron $P_2$ has exactly one proper simple subpolyhedron.
		\end{lemma}
		\begin{proof} 
			By means of the $\varepsilon$-invariant we show that $P_2$ has at least one proper simple subpolyhedron. If this is not the case, then $\mathcal{F}(P_2) = \{\emptyset, P_2\}$,
			$$
			t(P_3) = (-1)^{n} \varepsilon^{3-2n} + 1, \quad \text{and} \quad  t(P_2) = (-1)^{n-1} \varepsilon^{4-2n}  + 1.
			$$ 
			Since $t(M)=t(P_3)=t(P_2)$, we have
			$$
			\varepsilon^{4-2n} + \varepsilon^{3-2n} = \varepsilon^{3-2n}  (\varepsilon + 1) = \varepsilon^{3-2n} \varepsilon^{2} = \varepsilon^{5-2n} = 0,
			$$ 
			but the last equality contradicts the definition of $\varepsilon$. Thus, $P_2$ has at least one proper simple subpolyhedron.
		
			We show that there is exactly one such subpolyhedron. Let us assume the opposite. Since $P_2$ has exactly two $2$-components, it has at most two proper simple subpolyhedra, and each of these subpolyhedra contains exactly one $2$-component. Since $P_2$ is connected, it has an edge traversed twice by the boundary curve of one of the $2$-components and traversed once by the boundary curve of the other $2$-component. Hence, the latter $2$-component cannot be contained in a proper simple subpolyhedron of $P_2$. This completes the proof.
		\end{proof}	
	
		\begin{lemma} 
			\label{lm:properties} 
			Let $Q$ be the proper simple subpolyhedron of $P_2$ whose existence follows from Lemma \ref{lm:proper_subpolyh}. Then the following holds:
			\begin{itemize}
				\item[(i)] $\mathtt{v}(Q)$ and $n$ have the same parity. 
				\item[(ii)] $\chi(Q) - \mathtt{v}(Q) = 5-2n$.
				\item[(iii)] $\mathtt{v}(Q) \geqslant n-3$.
			\end{itemize}
		\end{lemma}
		\begin{proof} 
			Since $Q$ is a proper simple subpolyhedron of $P_2$, we have $\mathtt{v}(Q) < n-1$. Let us calculate the value of $t(M)$, using the polyhedron~$P_2$. Then
			$$
			t(P_2) = (-1)^{n-1} \varepsilon^{4- 2n} + (-1)^{\mathtt{v}(Q)} \varepsilon^{\chi(Q) - \mathtt{v}(Q)} + 1.  
			$$
			Since $t(P_3) = t(P_2)$, it follows that
			$$
			(-1)^{\mathtt{v}(Q)} \varepsilon^{\chi(Q) - \mathtt{v}(Q)} = (-1)^{n} \varepsilon^{5-2n}. 
			$$
			This implies that the first and the second statements of Lemma \ref{lm:properties} hold.
			
			Recall that $P_2$ has a natural cell decomposition: $2$-cells are $2$-components, $1$-cells are edges, and $0$-cells are true vertices. Denote the $2$-components of $P_2$ by $\alpha$ and $\beta$, assuming for definiteness that $\alpha$ is contained in $Q$, and $\beta$ is not. The cell decomposition of $P_2$ induces a cell decomposition of its simple subpolyhedron $Q$. Denote the number of $0$-cells of $Q$ by $k_{0}$ and the number of $1$-cells of $Q$ by $k_{1}$. Hence $k_{0} \leqslant n-1$ and $k_{1}\leqslant 2 k_{0}$. Since $Q$ has exactly one $2$-cell (that is the $2$-component $\alpha$), we have
			$$\chi(Q) = k_{0} - k_{1} + 1 \geqslant 1 - k_{0 }\geqslant 2-n.$$
			Applying this inequality to (ii), we get (iii). The proof is complete.
		\end{proof}			
		
		We return to the proof of the theorem and use the notations for $2$-components of $P_2$ as in the proof of Lemma \ref{lm:properties}. It follows from the statements (i) and (iii) of Lemma \ref{lm:properties} that $\mathtt{v}(Q) = n-2$. Then the boundary curve $\partial\beta$ passes only through one true vertex of $P_2$ and hence through only one edge of $P_2$. Moreover, this edge is a loop and $\partial\beta$ passes through it exactly once (see Fig.~\ref{fig:onewing}).
		
		\begin{figure}[ht]
			\begin{center}
				\includegraphics[scale=0.4]{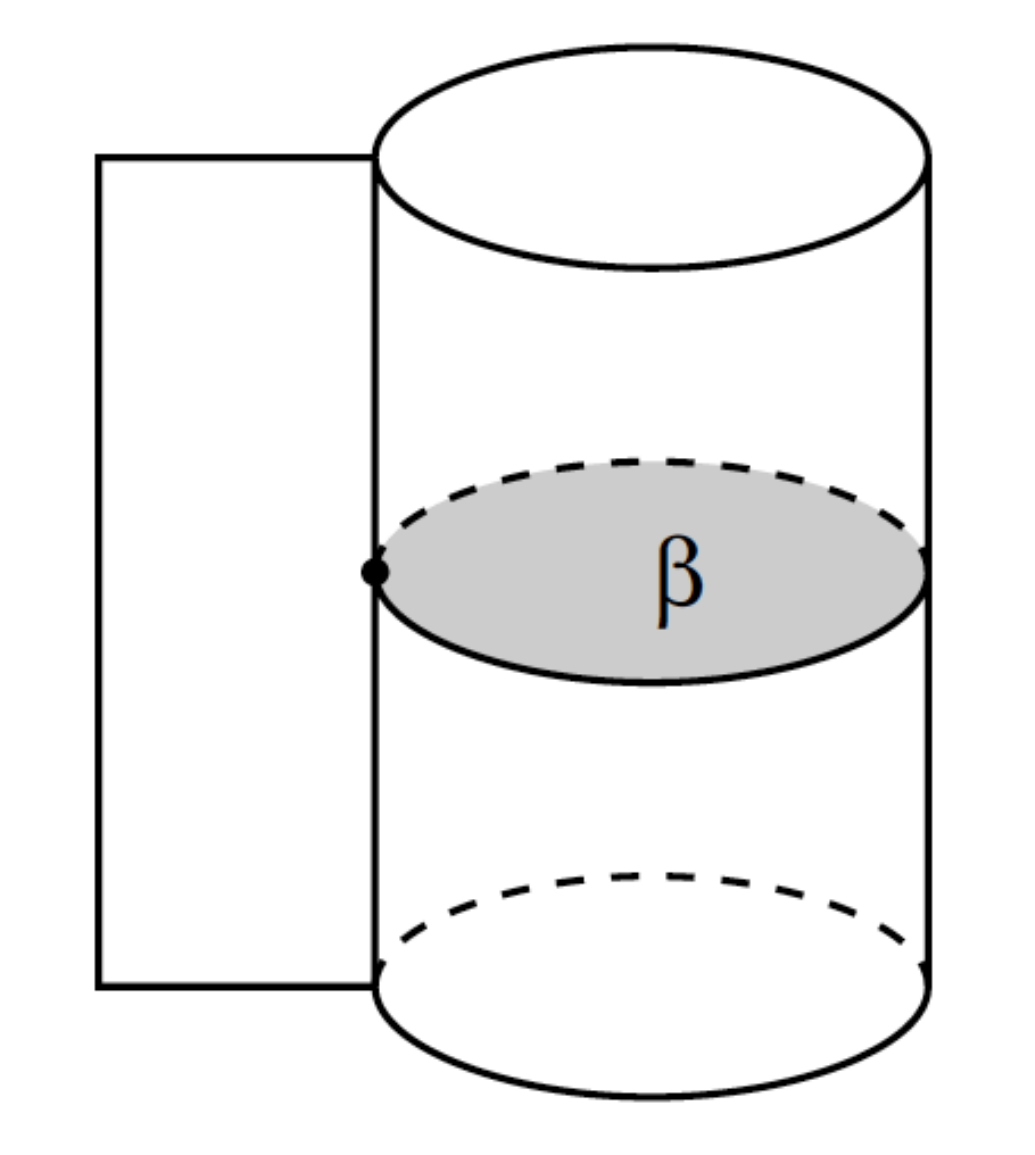}
			\end{center}
			\caption{The boundary curve $\partial\beta$ passes through one true vertex and one edge of the special polyhedron $P_2$.} 
			\label{fig:onewing}
		\end{figure}
	
		Hence $Q = P_2\setminus \beta$. So $\chi(Q) = \chi(P_2) - 1= 2-n$. Therefore, $\chi(Q) - \mathtt{v} (Q) = 4-2n$, which contradicts (ii). The proof is complete.
	\end{proof}
	
\section{Minimal triangulations of hyperbolic $3$-manifolds with totally geodesic boundary}
	
	We describe a family of ideal triangulations of $3$-manifolds with nonempty boundary satisfying the assumption of Theorem \ref{thm:main}. We say that a connected $4$-regular graph embedded in the plane is \emph{decorated} if at each vertex of the graph over- and underpasses are specified (just as for a crossing in a knot diagram), and each edge is assigned an element of the cyclic group $\mathbb Z_{3} = \{ 0,1,2 \}$. Decorated graphs are a partial case of the so-called o-graphs. It is known \cite[Theorem 0.1]{B-P} that every o-graph determines a special spine of a connected compact orientable $3$-manifold with nonempty boundary uniquely.

	Let $k, m, l$ be positive integers. Consider a planar connected $4$-regular graph $\Gamma (k, l, m)$ with exactly $2(k+l+m+3)$ vertices and with $2(k+l+m+3)$ double edges. Decompose the set of all its double edges into pairs such that the double edges in each pair are adjacent to the same vertex. We then decorate each such pair by one of the two possible ways as shown in Fig. \ref{fig:subgraphs} such that the final decoration corresponds to the cyclic string $AB^{k}AB^{l}AB^{m}$. For example, $\Gamma(2, 1, 1)$ is shown in Fig. \ref{fig:graph211}.
	
	\begin{figure}[ht]
		\begin{center}
			\includegraphics[scale=0.8]{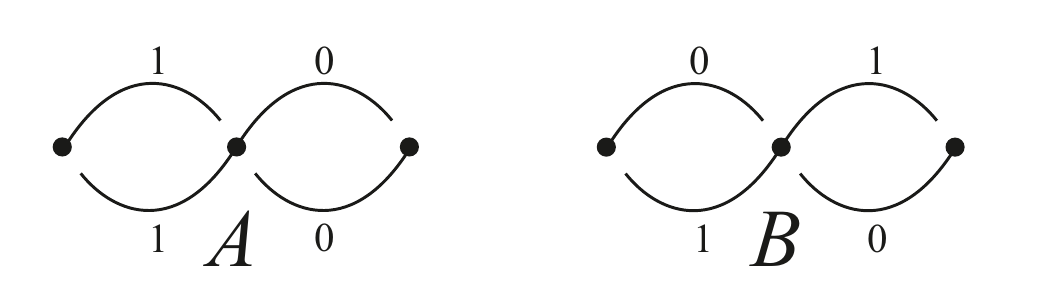}
		\end{center}
		\caption{Decorated subgraphs $A$ и $B$.} \label{fig:subgraphs}
	\end{figure}
	
	\begin{figure}[ht]
		\begin{center}
			\includegraphics[scale=0.8]{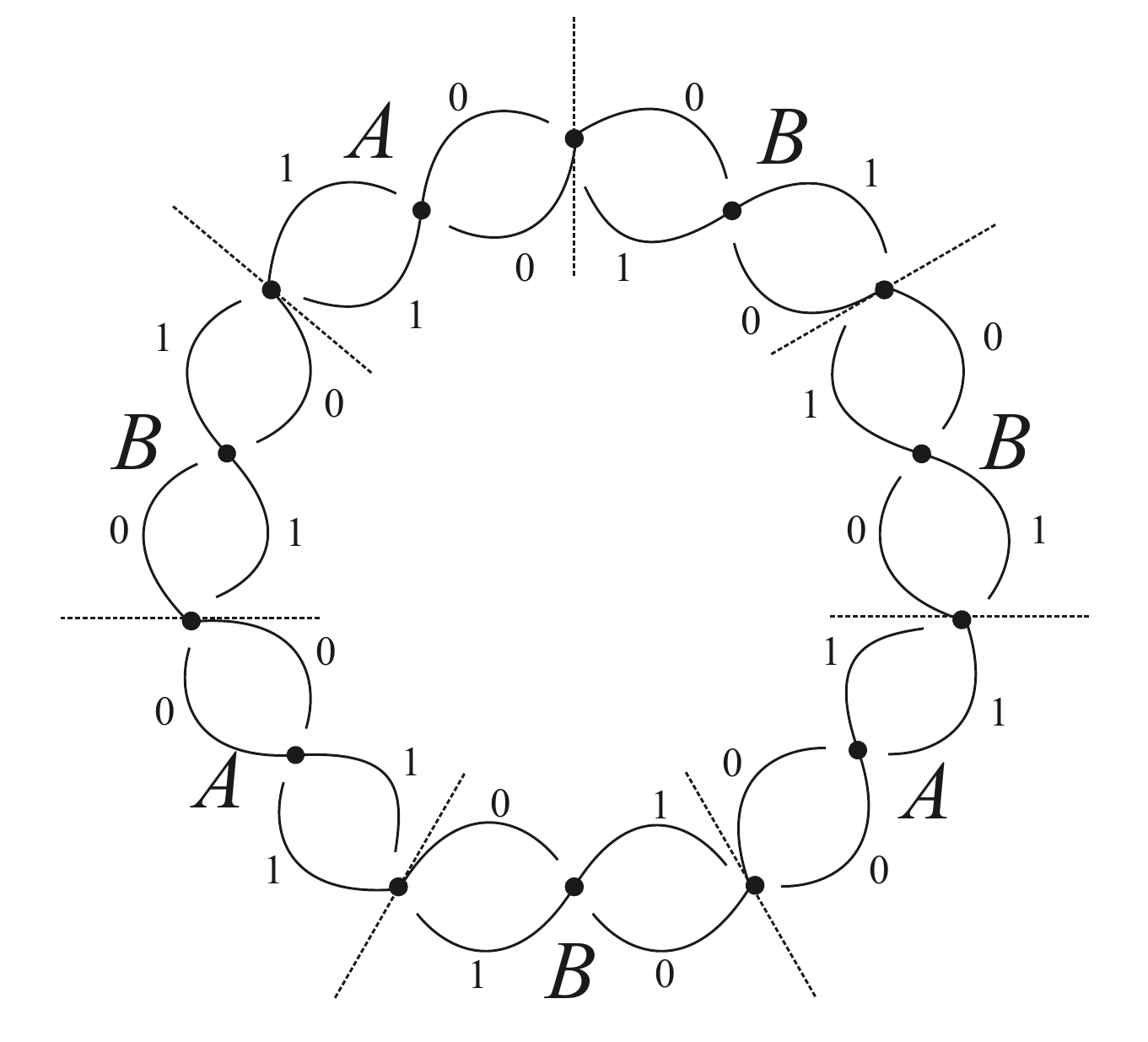}
		\end{center}
		\caption{Decorated graph $\Gamma(2, 1, 1)$.} 
		\label{fig:graph211}
	\end{figure}
	
	 According to \cite{B-P}, to obtain the special spine $P(k, l, m)$ determined by $\Gamma (k, l, m)$ we have to replace the subgraphs $A$ and $B$ in $\Gamma(k, l, m)$ with the similarly named blocks shown in Fig. \ref{fig:blocks_AB}. For example, $P (2, 1, 1)$ is shown in Fig. \ref{fig:spine211}. Let $W(k, l, m)$ be a connected compact orientable $3$-manifold with nonempty boundary determined by $P(k, l, m)$. Denote by $\mathcal{T}(k, l, m)$ the ideal triangulation of $W(k, l, m)$ dual to $P(k, l, m)$.
	 
	\begin{theorem} 
		\label{thm:example}
		For any positive integers $k, l, m$ the ideal triangulation $\mathcal{T}(k, l, m)$ of $W (k, l, m)$ is poor and three-edge.
	\end{theorem} 

	\begin{proof}
		It is easy to verify that for any positive integers $k$, $l$, and $m$ the special spine $P (k, l, m)$ contains exactly three $2$-components. Combinatorial analysis shows that $P (k, l, m)$ is poor. Then, by lemma \ref{lm:_poor_spine}, $\mathcal{T} (k, l, m)$ is poor and contains exactly three edges.
	\end{proof}

	 \begin{figure}[ht]
		\begin{center}
			\includegraphics[scale=0.7]{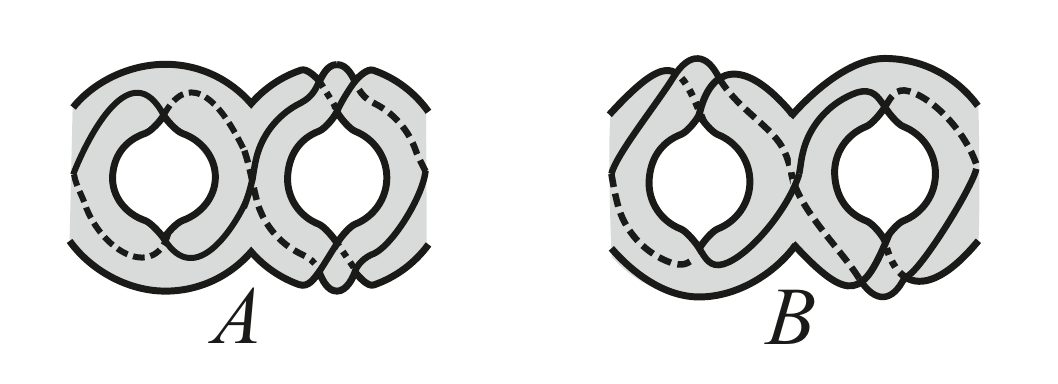}
		\end{center}
		\caption{Blocks $A$ и $B$.} 
		\label{fig:blocks_AB}
	\end{figure}
	
	\begin{figure}[ht]
		\begin{center}
			\includegraphics[scale=0.7]{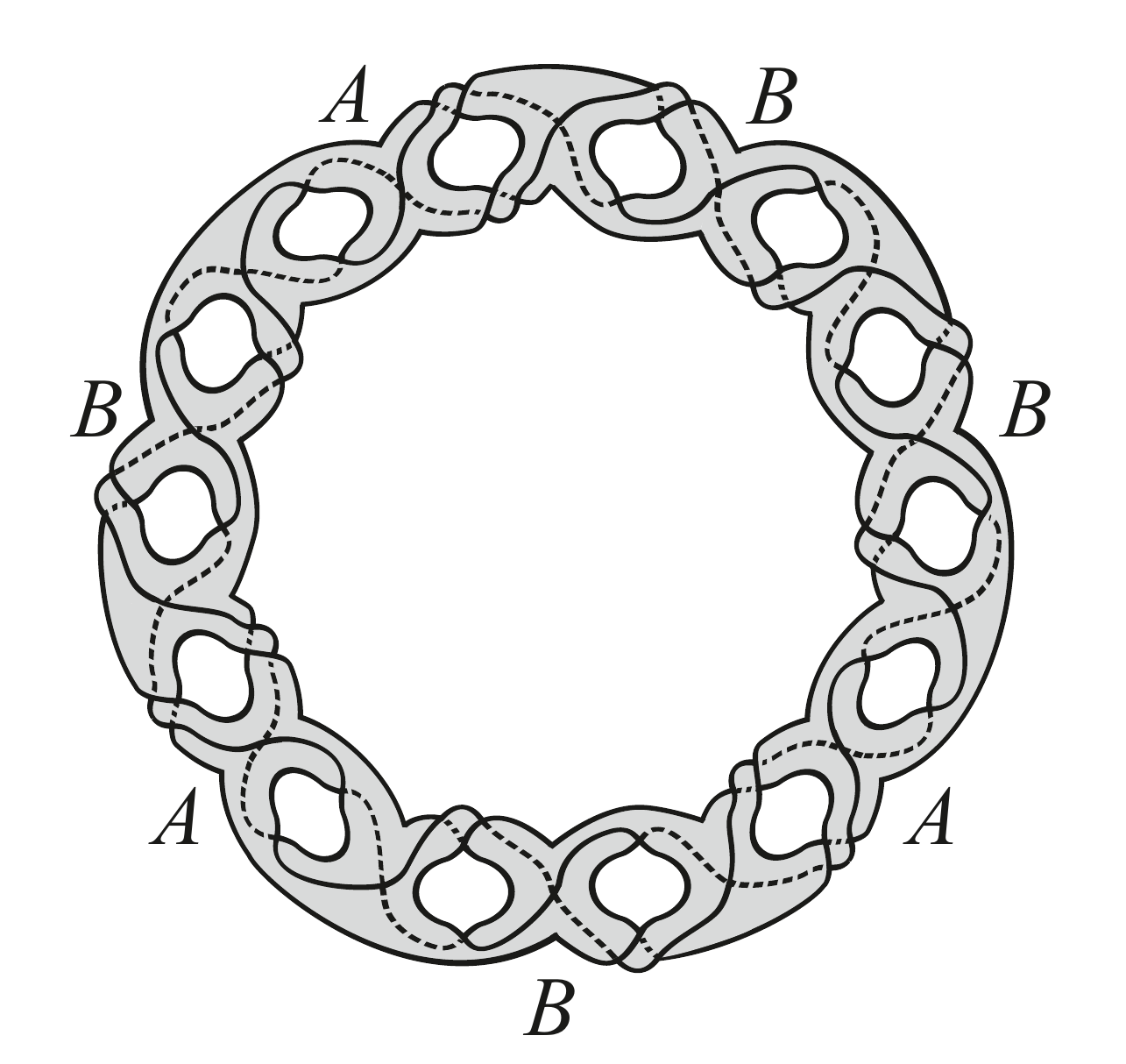}
		\end{center}
		\caption{Special spine $P(2, 1, 1)$.} 
		\label{fig:spine211}
	\end{figure}

\begin{theorem} 
	\label{thm:hyperbolic}
	For every positive integer $k$ the $3$-manifold $W(k, k, k)$ is hyperbolic with totally geodesic boundary.
\end{theorem} 

\begin{proof}
	Recall the notion of hyperbolic truncated tetrahedron \cite{FrigMarPet03-1}. Let $T$ be a tetrahedron and let $T^{*}$ be the combinatorial polyhedron obtained by removing from $T$ sufficiently small open stars of the vertices. We call lateral hexagon and truncation triangle the intersection of $T^{*}$ respectively with a face and with the link of a vertex of $T$. The edges of the truncation triangles are called  boundary edges, the other edges of $T^{*}$ are called internal edges. A hyperbolic truncated tetrahedron is a realization of $T^{*}$ as a compact polyhedron in $\mathbb H^{3}$, such that the truncation triangles are geodesic triangles, the lateral hexagons are geodesic hexagons, and the truncation triangles and lateral hexagons lie at right angles to each other. A truncated tetrahedron is regular if all the dihedral angles along its internal edges are equal to each other. For every $\theta$ with $0 < \theta < \pi/3$ there exists up to isometry exactly one regular truncated tetrahedron $T^{*}_{\theta}$ of dihedral angle $\theta$ \cite{Fujii}. The boundary edges of $T^{*}_{\theta}$ all have the same length, and the internal edges all have the same length.
	
	In the same way as in \cite{FrigMarPet03-1} in order to give $W(k, k, k)$ a hyperbolic structure, we identify each tetrahedron of $\mathcal{T}(k, k, k)$ with a copy of the regular truncated hyperbolic tetrahedron $T^{*}_{\pi/(6k+6)}$. Due to the symmetries of $T^{*}_{\pi/(6k+6)}$, every pairing between the model faces of the tetrahedra of $\mathcal{T}(k, k, k)$ can be realized by an isometry between the corresponding lateral hexagons of the $T^{*}_{\pi/(6k+6)}$’s.	
	By construction of $P(k, k, k)$, it is easy to see that the number of model edges that are identified to form an edge of $\mathcal{T} (k, k, k)$ is equal to $12(k+1)$ and does not depend on the choice of this edge. 
	Since the dihedral angles of $T^{*}_{\pi/(6k+6)}$ are equal to $\pi/(6k+6)$, we conclude that $W(k, k, k)$ is a hyperbolic $3$-manifold with totally geodesic boundary. 
\end{proof}

\end{document}